\documentclass[12pt]{article}

\usepackage{amsmath}
\usepackage{amssymb}
\usepackage{amsbsy,natbib}
\usepackage{amsthm}
\newtheorem{theorem}{Theorem}[section]
\newtheorem{definition}[theorem]{Definition}
\newtheorem{remark}{Remark}[section]
\newcommand{\Con}{{\mbox{C}}}
\newcommand{\R}{{\mathbb R}}

\title{Classical solvability of nonlinear initial-boundary problems
for first-order hyperbolic systems}

\author{
I.~Kmit 
\\
{\small
Institute for Applied Problems of Mechanics and Mathematics,}\\
{\small Ukrainian Academy of Sciences}\\
{\small Naukova St.\ 3b,
79060 Lviv,
Ukraine}
\\
{\small   E-mail:
{\tt kmit@informatik.hu-berlin.de}}
}

\date{}

\begin{document}

\maketitle

\begin{abstract}
We prove the global classical solvability of initial-boundary 
problems for semilinear
first-order hyperbolic systems subjected to local and nonlocal nonlinear
boundary conditions. We also establish lower bounds
for the order of nonlinearity demarkating a frontier between
regular cases (classical solvability) and singular  cases
(blow-up of solutions).
\end{abstract}

\section{Introduction}	

\noindent
We study existence, uniqueness, and continuous dependence on initial data
 of classical solutions to initial-boundary problems for semilinear hyperbolic
systems with nonlinear nonlocal boundary conditions.
Specifically,
in the domain $\Pi=\{(x,t)\,|\,0<x<1$, $t>0\}$
we address the following problem:
\begin{eqnarray}
(\partial_t  + \Lambda(x,t)\partial_x)
 u = f(x,t,u), \qquad\ (x,t)\in \Pi\ \quad\quad\label{eq:1}&&\\[2mm]
u(x,0) = \varphi(x), \qquad\qquad\qquad\qquad\quad\ \ \, x\in (0,1)\quad\label{eq:2}&&\\[1mm]
\begin{array}{ll}
u_i(0,t) = h_i(t,v(t)), & k+1\le i\le n, \ \  t\in(0,\infty)\\
u_i(1,t) = h_i(t,v(t)), & 1\le i\le k, \quad\quad\ \, t\in(0,\infty)
\end{array}&&\label{eq:3}
\end{eqnarray}
where  $u$, $f$, and $\varphi$  are real $n$-vectors,
$\Lambda=diag(\lambda_1,\dots,\lambda_n)$ is a diagonal matrix,
$\lambda_1,\dots,\lambda_k<0$, $\lambda_{k+1},\dots,\lambda_n>0$
for some $1\le k\le n$, and
$v(t)=(u_1(0,t),\dots,$
$u_{k}(0,t),u_{k+1}(1,t),\dots,u_{n}(1,t))$.
Note that the  system~(\ref{eq:1}) is non-strictly
hyperbolic and the boundary of $\Pi$ is non-characteristic.
We will denote $h=(h_1,\dots,h_n)$.

Special cases of~(\ref{eq:1})--(\ref{eq:3}) arise in
laser dynamics 
(Jochmann and Recke, 1999; Radziunas et al., 2000; Sieber, Recke and Schneider,
2004)
and chemical kinetics
(Zelenjak, 1966; Lyul'ko, 2002).

We establish a global existence-uniqueness classical result
for the problem~(\ref{eq:1})--(\ref{eq:3}).
Its novelty consists in allowing nonlinear local and nonlocal
boundary conditions and in allowing non-Lipschitz nonlinearities
in~(\ref{eq:1}) and~(\ref{eq:3}).
Namely, either the functions $f$ and $h$ can be both
non-Lipschitz with 
$\|f\|=O(\|u\|\log^{1/4}\log\|u\|)$ and
$\|g\|=O(\|v\|\log^{1/4}\log\|v\|)$ or
$f$ can be non-Lipschitz with 
$\|f\|=O(\|u\|\log\log\|u\|)$ while 
$g$ in this case should be Lipschitz
(a more detailed description is given in Section \ref{s:3},
see also Remark \ref{rem}).
Here and below by $\|\cdot\|$ we denote the Euclidian norm in $\R^n$.
It turns out that, in these conditions,
solutions to (\ref{eq:1})--(\ref{eq:3}) have the same
qualitative behaviour as in the linear case.

Our result should be contrasted to the known fact (see, e.g., 
Li Ta-tsien, 1994; Alinhac, 1995)
that  if the right hand side   $f(x,t,u)$
of (\ref{eq:1})  has at least quadratic growth in $u$, then
classical solutions to (\ref{eq:1})--(\ref{eq:3}) in general fail to exist
globally in time. More precisely, they blow-up in a finite time,
creating singularities.
Thus, we show that a ``frontier'' between growth rates of $f$
ensuring regular behavior of the system and causing singular behavior
lies somewhere between $\|u\|\log\log\|u\|$ and $\|u\|^2$.
An intriguing problem is to make the gap closer.

For different aspects of the subject we refer the reader to
Myshkis and Filimonov (1981, 2003) and to Kmit (2006a, 2006b).
Myshkis and Filimonov (1981) investigate problems
with nonlinear boundary conditions, but only of the {\it local} type.
In contrast to this,
our case includes nonlinear {\it nonlocal} boundary conditions.
Furthermore, Myshkis and Filimonov (1981) 
make an essential assumption on nonlinearities, namely, that
$f$ and $h$ are globally Lipshitz in, respectively, $u$ and $v$. This means that
$f$ (resp., $g$) admits no more than linear growth in $u$ (resp., $v$)
as $\|u\|\to\infty$ (resp.,  $\|v\|\to\infty$). 

Kmit (2006a, 2006b) considers the problem~(\ref{eq:1})--(\ref{eq:3}) admitting
strong singularities both in the differential equations and in the
initial-boundary conditions. The author proves a general existence-uniqueness 
result in the Colombeau algebra of generalized functions and derives asymptotic
estimates for generalized solutions.

\section{The Case of  Lipschitz Nonlinearities:
Existence, Uniqueness, and Continuous Dependence}\label{s:2}

\noindent
If the initial data of the problem~(\ref{eq:1})--(\ref{eq:3}) are
sufficiently smooth, then
the zero-order and the first-order compatibility conditions
between~(\ref{eq:1}) and~(\ref{eq:2}) are given by equalities
\begin{equation}\label{eq:c1}
\begin{array}{lccr}
\varphi_i(0)&=&h_i(0,v(0)),&\quad k+1\le i\le n,\\
\varphi_i(1)&=&h_i(0,v(0)),&\quad 1\le i\le k,
\end{array}
\end{equation}
and
\begin{equation}\label{eq:c2}
\begin{array}{l}
f_i(0,0,\varphi(0))-\lambda_i(0,0)\varphi_i'(0)=
\partial_th_i(0,v(0))\\
\qquad\qquad\quad+\nabla_vh_i(0,v(0))\cdot v'(0),\quad k+1\le i\le n;\\[2mm]
f_i(1,0,\varphi(1))-\lambda_i(1,0)\varphi_i'(1)=
\partial_th_i(0,v(0))\\
\qquad\qquad\quad+\nabla_vh_i(0,v(0))\cdot v'(0),\quad 1\le i\le k,
\end{array}
\end{equation}
where
\begin{eqnarray*}
v(0)&=&\Bigl(\varphi_1(0),\dots,\varphi_k(0),\varphi_{k+1}(1),\dots,
\varphi_{n}(1)\Bigr),
\\[1mm]
v'(0)&=&\Bigl(f_1(0,0,\varphi(0))-\lambda_1(0,0)\varphi_1'(0),\dots,\\&&
f_k(0,0,\varphi(0))-\lambda_k(0,0)\varphi_k'(0),\\&&
f_{k+1}(1,0,\varphi(1))-\lambda_{k+1}(1,0)\varphi_{k+1}'(1),\\&&
\dots,f_n(1,0,\varphi(1))-\lambda_n(1,0)\varphi_n'(1)\Bigr),
\end{eqnarray*}
and $"\cdot"$ denotes the scalar product in $\R^n$.

\begin{theorem}\label{thm:classical}
 Assume that the initial data $\lambda_i$ and $f_i$ are continuous,
$\lambda_i$ and $\varphi_i$ are $\Con^1$-smooth in $x$, $f_i$ are
$\Con^1$-smooth in $x$ and $u$, $h_i$ are $\Con^1$-smooth in both arguments.
Let  $\nabla_yf(x,t,y)$ be bounded on $K\times\R^n$ for every
compact  $K\subset\overline{\Pi}$ and
 $\nabla_zh(t,z)$ be
bounded on $K\times\R^n$
for every compact $K\subset[0,\infty)$. If the zero-order and the first-order
compatibility conditions~(\ref{eq:c1}) and~(\ref{eq:c2}) are
fulfilled, then
the problem~(\ref{eq:1})--(\ref{eq:3}) has a unique classical solution in
$\Pi$.
\end{theorem}

\begin{proof}
An equivalent integral-operator representation of~(\ref{eq:1})--(\ref{eq:3})
can be written in the form
\begin{eqnarray}
\lefteqn{u_i(x,t) = (R_iu)(x,t)}\nonumber\\&&
+
\int\limits_{t_i(x,t)}^t\biggl[u(\omega_i(\tau;x,t),\tau)\cdot
\int\limits_0^1\nabla_uf_i(\omega_i(\tau;x,t),\tau,\sigma u)\,d\sigma
\nonumber\\&&
+f_i(\omega_i(\tau;x,t),\tau,0)\biggr]\,d\tau,\qquad
1\le i\le n,\label{eq:integral}
\end{eqnarray}
where
\begin{multline*}
\shoveright{\,\,(R_iu)(x,t) = \varphi_{i}(\omega_i(0;x,t))}
\end{multline*}
if $t_i(x,t)=0$ and
\begin{eqnarray*}
\lefteqn{(R_iu)(x,t)=}\\&&
=
v(t_i(x,t))\cdot\int\limits_0^1\nabla_v
h_i(t_i(x,t),\sigma v)\,d\sigma+h_i(t_i(x,t),0)
\end{eqnarray*}
otherwise. Here
$\omega_i(\tau;x,t)$ denotes the $i$-th characteristic of~(\ref{eq:1})
passing through $(x,t)\in\overline{\Pi}$ and  $t_i(x,t)$  denotes
the smallest value of
$\tau\ge 0$ at which
$\xi=\omega_i(\tau;x,t)$ riches $\partial\Pi$.
Given $T>0$, denote
$$
\Pi^T=\{(x,t)\,|\,0<x<1, 0<t<T\}.
$$
It suffices to prove the theorem in $\Pi^T$ for an arbitrarily fixed $T>0$.
Let  $L_f$ be a Lipshitz constant of $f_i(x,t,u)$ in $u$
which is uniform in $i\le n$ and $(x,t)\in\overline{\Pi^T}$,
$L_h$ be a Lipschitz constant of $h_i(t,v)$ in $v$
which is uniform in $i\le n$ and $t\in[0,T]$.

We split our argument into two claims.
In parallel we will derive global a
priori estimates, which will be used in the next section.

{\it Claim 1. (\ref{eq:integral}) has  a unique  continuous
solution in $\overline{\Pi^T}$.}
We  first prove that there exists a unique solution
$u\in(C(\overline{\Pi^{\theta_0}}))^n$ to~(\ref{eq:integral})
for some   $\theta_0>0$ such that

\begin{equation}\label{eq:22_0}
\omega_i(t;0,\tau)<\omega_j(t;1,\tau)
\end{equation}
for all
$$
\tau\ge 0,  t\in[\tau,\tau+\theta_m],
 k+1\le i\le n, 1\le j\le k,
$$
where $m=0$.
For $t\in[0,\theta_0]$ we can express
$v(t)$ in the form

\begin{eqnarray}
\lefteqn{v_i(t)=\varphi_{i}(\omega_i(0;x_i,t))}\nonumber\\&&
+
\int\limits_{0}^t\biggl[u(\omega_i(\tau;x_i,t),\tau)\cdot\int\limits_0^1\nabla_u
f_i(\omega_i(\tau;x_i,t),\tau,\sigma u)\,d\sigma\nonumber
\\
&&+f_i(\omega_i(\tau;x_i,t),\tau,0)\biggr]\,d\tau,\qquad
1\le i\le n,\label{eq:C}
\end{eqnarray}
where $x_i=0$ for $1\le i\le k$ and $x_i=1$ for $k+1\le i\le n$.

\medskip

{\it Convention.}
In the maximization operators below, unless their range is explicitly
specified, we assume the following:

\hangafter=0
\hangindent=\parindent
\noindent
the range of $i,x$ is $i\le n$, $x\in[0,1]$;\\
the range of $i,t$ is $i\le n$, $t\in[0,T]$;\\
the range of $i,x,t$ is $i\le n$, $x\in[0,1]$, $t\in[0,T]$;\\
the range of $i,t,z$ is $i\le n$, $t\in[0,T]$, $\|z\|\le M$;\\
the range of $i,x,t,y$ is $i\le n$, $x\in[0,1]$, $t\in[0,T]$, $\|y\|\le M$,
where a constant $M$ will be specified later.

\medskip

Apply the contraction mapping principle to~(\ref{eq:integral}).
Applying the operator defined by the right hand side of (\ref{eq:integral})
to continuous functions $u^1$ and $u^2$ and
considering the difference  $u^1-u^2$ in
$\overline{\Pi^{\theta_0}}$, we get
$$
\max\limits_{i\le n;(x,t)\in\overline{\Pi^{\theta_0}}}|u_i^1-u_i^2|\le
\theta_0q_0\max\limits_{i\le n;(x,t)\in\overline{\Pi^{\theta_0}}}|u_i^1-u_i^2|,
$$
where
\begin{eqnarray*}
q_0=nL_f(1+nL_h).
\end{eqnarray*}
Choose
$$
\theta_0=\left(2q_0\right)^{-1}.
$$
This proves the
existence and uniqueness of a $(C(\overline{\Pi^{\theta_0}}))^n$-solution $u$,
satisfying
the following
local a priori estimate:
\begin{equation}\label{eq:24}
\max\limits_{i\le n;(x,t)\in\overline{\Pi^{\theta_0}}}|u_i|\le
2\left(1+nL_h\right)\Phi,
\end{equation}
where
\begin{equation}\label{eq:phi}
\Phi=\max\limits_{i,x}|\varphi_i(x)|+T\max\limits_{i,x,t}|f_i(x,t,0)|
+\max\limits_{i,t}|h_i(t,0)|.
\end{equation}
Note that  the value of $q_0$ depends on $T$ and does not depend on $\theta_0$.
This allows us to complete the proof of the claim
in $\lceil T/\theta_0\rceil$ steps,
iterating the local existence-uniqueness result in domains
$
(\Pi^{j\theta_0}\cap\Pi^T)\setminus\overline{\Pi^{(j-1)\theta_0}}
$,
where
$
j\le \lceil T/\theta_0\rceil.
$
Simultaneously we arrive at the global a priori estimate
\begin{equation}\label{eq:u}
\max\limits_{i,x,t}|u_i|\le
\left(3+2nL_h\right)^{\lceil T/\theta_0\rceil}\Phi.
\end{equation}

{\it Claim 2. (\ref{eq:1})--(\ref{eq:3}) has  a unique
$\Con^1$-solution in $\Pi^T$.}
We start with a problem for $\partial_xu$:
\begin{eqnarray}
\lefteqn{\partial_xu_i(x,t) = (R_{ix}'u)(x,t)}\nonumber\\&&
\displaystyle
\!\!\!\!+\int\limits_{t_i(x,t)}^t\biggl[\nabla_uf_i(\xi,\tau,u)\cdot\partial_\xi u(\xi,\tau)
-\partial_\xi\lambda_i(\xi,\tau)\partial_\xi u_i(\xi,\tau)\nonumber\\&&
\!\!\!\!+\left(\partial_\xi f_i\right)(\xi,\tau,u)\biggr]
\bigg|_{\xi=\omega_i(\tau;x,t)}
\,d\tau,\qquad 1\le i\le n,\label{eq:U_x}
\end{eqnarray}
where
\begin{multline*}
\shoveright{\ (R_{ix}'u)(x,t)=\varphi_{i}'(\omega_i(0;x,t))}
\end{multline*}
if $t_i(x,t)=0$ and
\begin{multline*}
\ (R_{ix}'u)(x,t)=
\lambda_i^{-1}(y_i,\tau)
\Bigl[
f_i(y_i,\tau,u)\\
-\nabla_vh_i(\tau,v)\cdot v'(\tau)-
(\partial_th_i)(\tau,v)\Bigr]\Big|_{\tau=t_i(x,t)}
\end{multline*}
otherwise. Here
 $y_i=0$ for $k+1\le i\le n$ and
$y_i=1$ for $1\le i\le k$.
Let us show that there is a unique solution $\partial_xu\in\Con\left(\overline{\Pi^{\theta_1}}\right)$ 
to~(\ref{eq:U_x}) for some $\theta_1$ satisfying the condition~(\ref{eq:22_0}) with $m=1$.
Combining~(\ref{eq:1}) with~(\ref{eq:U_x}) for  $t\in[0,\theta_1]$,  we get
\begin{eqnarray}
\lefteqn{v_i'(t)=f_i(x_i,t,u)-\lambda_i(x_i,t)\partial_xu_i(x_i,t)}\nonumber\\&&
=f_i(x_i,t,u)-\lambda_i(x_i,t)\biggl[\varphi_{i}'(\omega_i(0;x_i,t))\nonumber\\&&
\displaystyle
+
\int\limits_{0}^t\Bigl[\nabla_uf_i(\xi,\tau,u)\cdot\partial_\xi u(\xi,\tau)
-\partial_\xi\lambda_i(\xi,\tau)\partial_\xi u_i(\xi,\tau)\nonumber\\&&
+\left(\partial_\xi f_i\right)(\xi,\tau,u)\Bigr]
\Big|_{\xi=\omega_i(\tau;x_i,t)}
\,d\tau\biggr],\ \ 1\le i\le n.\label{eq:v'}
\end{eqnarray}
Using the fact that $u$ is a known continuous function (see Claim~1),
we now  apply the operator defined by the right hand side of~(\ref{eq:U_x})
to  continuous functions $\partial_xu^1$ and $\partial_xu^2$.
Notice the estimate
$$
\max\limits_{i\le n;(x,t)\in\overline{\Pi^{\theta_1}}}
|\partial_xu_i^1-\partial_xu_i^2|\le
\theta_1q_1\max\limits_{i\le n;(x,t)\in\overline{\Pi^{\theta_1}}}
|\partial_xu_i^1-\partial_xu_i^2|,
$$
where
$$
q_1=\left(nL_f+\max\limits_{i,x,t}
|\partial_x\lambda_i|\right)
\left(1+nL_h\max\limits_{i,x,t}|\lambda_i|\max\limits_{i,x,t}|\lambda_i|^{-1}\right).
$$
Choose
$$
\theta_1=(2q_1)^{-1}.
$$
This shows that the operator
defined by the right hand side of (\ref{eq:U_x}) has the
contraction property with respect to the domain $\overline{\Pi^{\theta_1}}$
and proves the existence and the uniqueness of
$u\in \Con_{x,t}^{1,0}(\overline{\Pi^{\theta_1}})$.
Furthermore,


\begin{equation}\label{eq:U1}
\max\limits_{i\le n;(x,t)\in\overline{\Pi^{\theta_1}}}|\partial_xu_i|
\le 2\left(1+nL_h
\max\limits_{i,x,t}|\lambda_i|\max\limits_{i,x,t}|\lambda_i|^{-1}
\right)\Psi,
\end{equation}
where
\begin{equation}\label{eq:Psi}
\begin{array}{ccc}
\displaystyle
\Psi=\max\limits_{i,x}|\varphi_{i}'|
+T\max\limits_{i,x,t,y}|\partial_xf_i|
\\[3mm]\displaystyle
+\max\limits_{i,x,t}|\lambda_i|^{-1}
\max\limits_{i,x,t,y}|f_i|
+\max\limits_{i,x,t}|\lambda_i|^{-1}
\max\limits_{i,t,z}|\partial_th_i|
\end{array}
\end{equation}
and the constant $M$ introduced above in Convention is now set up to
$$
M=n\left(3+2nL_h\right)^{\lceil T/\theta_0\rceil}\Phi
$$
(see the estimate~(\ref{eq:u})).
Note that $q_1$ depends on $T$ and does not on $\theta_1$.
To complete the proof of the claim, it hence remains to
iterate the local existence-uniqueness result in domains
$
(\Pi^{j\theta_1}\cap\Pi^T)\setminus\overline{\Pi^{(j-1)\theta_1}}
$,
where
$
j\le \lceil T/\theta_1\rceil.
$
This also gives us the global a priori estimate
\begin{equation}\label{eq:apr1}
\max\limits_{i,x,t}|\partial_xu_i|\le
\left(3+2nL_h
\max\limits_{i,x,t}|\lambda_i|\max\limits_{i,x,t}|\lambda_i|^{-1}\right)^{\lceil T/\theta_1\rceil}\Psi.
\end{equation}
The fact that
$u$ is a $\Con^1$-function in both arguments follows now from~(\ref{eq:1}).
Furthermore,
\begin{equation}\label{eq:apr11}
\max\limits_{i,x,t}|\partial_tu_i|\le \max\limits_{i,x,t,y}
|f_i|+\max\limits_{i,x,t}|\lambda_i|\max\limits_{i,x,t}|\partial_xu_i|,
\end{equation}
where $\partial_xu_i$ satisfy~(\ref{eq:apr1}). The claim is proved.

Since $T$ is arbitrary, the theorem follows.
\end{proof}

\begin{definition}\label{defn:continuous}
A continuous solution to the integral-operator system~(\ref{eq:integral})
is called a continuous solution to the problem~(\ref{eq:1})--(\ref{eq:3}).
\end{definition}

From the proof of Claim~1 (in the proof of Theorem \ref{thm:classical}) we obtain also
the following fact: If all the initial data
in (\ref{eq:1})--(\ref{eq:3}) are continuous functions and
$f_i$ and $h_i$ are globally Lipschitz, respectively, in $u$ and $v$,
then there is a unique continuous solution to~(\ref{eq:1})--(\ref{eq:3})
satisfying the global a priori estimate~(\ref{eq:u}).
This gives us the following continuous dependence theorem.

\begin{theorem}\label{thm:continuous}
 Assume that the initial data $\lambda_i$, $f_i$, $\varphi_i$, and $h_i$ are
continuous functions in their arguments and $\lambda_i$ are
Lipschitz in $x\in[0,1]$.
Let  $\nabla_yf(x,t,y)$ be bounded on $K\times\R^n$ for every
compact  $K\subset\overline{\Pi}$ and
 $\nabla_zh(t,z)$ be
bounded on $K\times\R^n$
for every compact $K\subset[0,\infty)$. Suppose that the zero-order
compatibility conditions~(\ref{eq:c1}) are
fulfilled. If $f(x,t,0)\equiv 0$ for all $(x,t)\in\overline\Pi$ and
$h(t,0)\equiv 0$ for all $t\in[0,\infty)$,
 then the continuous solution to the
problem~(\ref{eq:1})--(\ref{eq:3}) continuously depends on $\varphi(x)$.
\end{theorem}


\section{The Case of  Non-Lipschitz Nonlinearities:
Existence and Uniqueness Result}\label{s:3}

\noindent
We here extend Theorem \ref{thm:classical} to the case of
non-Lipschitz nonlinearities in~(\ref{eq:1}) and~(\ref{eq:3}).
\begin{theorem}\label{thm:nonL}
 Assume that the initial data $\lambda_i$ and $f_i$ are continuous,
$\lambda_i$ and $\varphi_i$ are $\Con^1$-smooth in $x$, $f_i$ are
$\Con^1$-smooth in $x$ and $u$, $h_i$ are $\Con^1$-smooth in both arguments.
Suppose that for each $T>0$ there exist $C_f>0$ and $C_h>0$ such that
\begin{equation}\label{eq:f}
\left\|\nabla_yf(x,t,y)\right\|\le C_f\left(\log\log F(x,t,\|y\|)\right)^{1/4},
\end{equation}
\begin{equation}\label{eq:h}
\left\|\nabla_zh(t,z)\right\|\le C_h\left(\log\log H(t,\|z\|)\right)^{1/4},
\end{equation}
where $F$ (resp., $H$) is a polynomial in $\|y\|$ (resp., in $\|z\|$) with
coefficients in $\Con^1\left(\overline{\Pi^T}\right)$
(resp., in $\Con^1[0,T]$).
If the zero-order and the first-order
compatibility conditions~(\ref{eq:c1}) and~(\ref{eq:c2}) are
fulfilled, then  the problem~(\ref{eq:1})--(\ref{eq:3}) has a unique
classical solution in $\Pi$.
\end{theorem}

\begin{proof}
It suffices to prove the theorem in $\Pi^T$ for an arbitrarily fixed $T>0$.
Let us prove that there exists a unique continuous solution to our problem
(in the sense of Definition~\ref{defn:continuous}) such that
\begin{equation}\label{eq:star}
\max\limits_{i,x,t}|u_i|\le e^R/\sqrt{n}
\end{equation}
for all sufficiently large $R>0$. On the account of~(\ref{eq:u}),
we are done if we show that
\begin{multline}\label{eq:final}
\Phi\left[3+2nC_h\left(\log\log\max\limits_{[0,T]\times[0,e^R]}H(t,\|z\|)\right)^{1/4}
\right]^{\lceil T/\theta_0\rceil}\\[1mm]
\le e^R/\sqrt{n},
\end{multline}
where
\begin{eqnarray*}
\displaystyle
\theta_0=\left(2C_fn\right)^{-1}
\left(\log\log\max\limits_{\overline{\Pi^T}\times[0,e^R]}F(x,t,\|y\|)
\right)^{-1/4}
\\
\displaystyle
\times\left(1+nC_h\left(\log\log\max\limits_{[0,T]\times[0,e^R]}H(t,\|z\|)
\right)^{1/4}\right)^{-1}.
\end{eqnarray*}
Let $\sigma$ be the largest maximum absolute value of coefficients of $F$ 
and $H$ in $\overline{\Pi^T}$. Let $\delta$ be the maximum degree of the
polynomials $F$ and $H$. Set
$$
S=\sigma\left(1+e^R\right)^\delta.
$$
It is easy to see that
$$
\max\left\{\max\limits_{\overline{\Pi^T}\times[0,e^R]}F(x,t,\|y\|),
\max\limits_{[0,T]\times[0,e^R]}H(t,\|z\|)\right\}\le S.
$$
Obviously, there exists $R_0>0$ such that for all
$R\ge R_0$ the left hand side of~(\ref{eq:final})
is  bounded from above by
\begin{eqnarray*}
\Phi\left[\log\log S\right]^{1/2\left(\log\log S\right)^{1/2}}.
\end{eqnarray*}
Fix an arbitrary $R\ge R_0$ so that
$$
\Phi\left[(1+\delta)\log(2\sigma)+\delta R\right]\le e^R/\sqrt{n}.
$$
The desired estimate~(\ref{eq:final}) now follows from the inequality
\begin{eqnarray}
\lefteqn{\Phi\left[\log\log S\right]^{1/2\left(\log\log S\right)^{1/2}}\le}
\nonumber\\&&
\le
\Phi\exp\left\{1/2\log\left(\log\log S\right)\left(\log\log S\right)^{1/2}\right\}
\nonumber\\&&
=\Phi\exp\left\{\log\left(\log\log S\right)^{1/2}\left(\log\log S\right)^{1/2}\right\}
\nonumber\\&&
\le\Phi\exp\left\{\log\log S\right\}
=\Phi\log S
\le
\Phi\log\left(\sigma(1+e^R)^\delta\right)\nonumber\\&&
=
\Phi\left[(1+\delta)\log(2\sigma)+\delta R\right]
\le e^R/\sqrt{n}.\label{eq:calcul}
\end{eqnarray}
The existence and the uniqueness of a continuous solution
satisfying the bound (\ref{eq:star}) is therewith proved.
This gives us the unconditional existence and, since $R\ge R_0$
is arbitrary, we have also the unconditional uniqueness.

To prove that the solution is a
$\left[\Con_{x,t}^{1,0}(\Pi^T)\right]^n$-function, we apply a similar argument, but now 
use the global a priori estimate~(\ref{eq:apr1}). It suffices to show that for 
some $Q>0$ there is a unique continuous function $\partial_xu$ with
$$
\max\limits_{i,x,t}|\partial_xu_i|\le e^Q/\sqrt{n}.
$$
Fix $P\ge R_0$ and set up the constant $M$ introduced by Convention
in Section \ref{s:2} to $M=e^P/\sqrt{n}$.
Notice the existence of a constant $Q_0>0$ such that
for all $Q\ge Q_0$ the right hand side of~(\ref{eq:apr1}) is bounded from above by
\begin{eqnarray*}
\Psi\left[\log\log S\right]^{1/2\left(\log\log S\right)^{1/2}}
\end{eqnarray*}
and choose  $Q\ge Q_0$ satisfying the inequality
$$
\Psi\left[(1+\delta)\log(2\sigma)+\delta Q\right]\le e^Q/\sqrt{n}.
$$
To finish the proof of the $\left[\Con_{x,t}^{1,0}(\Pi^T)\right]^n$-smoothness,
it remains to apply the calculation~(\ref{eq:calcul}) with $\Psi$ in place of $\Phi$.

By~(\ref{eq:1}), $u$ is a $\left[\Con^1(\Pi^T)\right]^n$-function.
The proof is complete.
\end{proof}

\begin{remark}\rm
To prove the uniqueness part of Theorem~\ref{thm:nonL}, one can also run a
standard argument. Let $u$ and $w$ be two classical solutions to the
problem (\ref{eq:1})--(\ref{eq:3}). Then $u-w$ satisfies 
the sytem
\begin{eqnarray*}
&&(\partial_t  + \Lambda(x,t)\partial_x) u =\\ 
&&\qquad\quad=(u-w)\cdot
\int\limits_0^1\nabla_uf\left(x,t,\sigma u+(1-\sigma)w\right)\,d\sigma,\\
&&u(x,0) = 0,\\
&&u_i(0,t) = (v-\tilde v)\cdot\int\limits_0^1
\nabla_vh_i\left(x,t,\sigma v+(1-\sigma)\tilde v\right)\,d\sigma,\\ 
&&\qquad\qquad\qquad\qquad\qquad\qquad\qquad k+1\le i\le n;\\
&&u_i(1,t) = (v-\tilde v)\cdot\int\limits_0^1
\nabla_vh_i\left(x,t,\sigma v+(1-\sigma)\tilde v\right)\,d\sigma,\\
&&\qquad\qquad\qquad\qquad\qquad\qquad\qquad 1\le i\le k.
\end{eqnarray*}
 Here
$$\tilde v(t)=(w_1(0,t),\dots,w_{k}(0,t),w_{k+1}(1,t),\dots,w_{n}(1,t)).$$ Since
$\nabla_uf(x,t,\sigma u+(1-\sigma)w)$ and 
$\nabla_vh(x,t,\sigma v+(1-\sigma)\tilde v)$ are known continuous functions,
the uniqueness now follows from an analog of (\ref{eq:u}) for the 
difference $u-w$.
\end{remark}

\begin{remark}\label{rem}\rm
Assume that all conditions excluding~(\ref{eq:f}) and~(\ref{eq:h})
of Theorem~\ref{thm:nonL} are fulfilled. Furthermore, assume that
$\nabla_zh(t,z)$ is
bounded on $K\times\R^n$
for every compact $K\subset[0,\infty)$ and
for each $T>0$ there exists $C_f>0$ such that
$$
\left\|\nabla_yf(x,t,y)\right\|\le C_f\log\log F(x,t,\|y\|),
$$
where $F$ is a polynomial as in Theorem~\ref{thm:nonL}. Then, using similar
argument as in the proof of Theorem~\ref{thm:nonL}, one can easily prove that
the problem~(\ref{eq:1})--(\ref{eq:3}) has a unique
classical solution in $\Pi$.
\end{remark}

\bibliographystyle{apalike}

\end{document}